\documentclass[10pt]{amsart}

\usepackage{a4wide}
\usepackage{amsmath}
\usepackage{amsfonts}
\usepackage{amssymb}
\usepackage{yfonts}
\usepackage{epsf}
\usepackage{graphicx}
\usepackage{colonequals}
\usepackage{color}
\usepackage{algorithm,algorithmic,color,amsmath}
\usepackage{tabularx}
\usepackage[percent]{overpic}
\usepackage{hyperref,amsbsy}

\newtheorem{theorem}{Theorem}[section]
\newtheorem{lemma}[theorem]{Lemma}

\newtheorem{corollary}[theorem]{Corollary}

\theoremstyle{definition}
\newtheorem{definition}[theorem]{Definition}

\newtheorem{conjecture}[theorem]{Conjecture}
\theoremstyle{remark}
\newtheorem{remark}[theorem]{Remark}

\numberwithin{equation}{section}

\newcommand{\N}{\mathbb{N}}

\newcommand{\R}{\mathbb{R}}

\newcommand {\oesgn}[1] {{\rm oesgn}(#1)}
\newcommand {\sgn}[1] {{\rm sgn}(#1)}

\title[Polynomials with bounded roots]{Distribution results on polynomials with bounded roots}

\dedicatory{Dedicated to Professor Johann Cigler on the occasion of his 80th birthday}

\author [P.~Kirschenhofer and J.~Thuswaldner]{Peter~Kirschenhofer and J\"org~Thuswaldner}
\address{Chair of Mathematics and Statistics,
Montanuniversit\"at Leoben, Franz Josef-Stra\ss{}e 18, A-8700 Leoben,
AUSTRIA} \email{Peter.Kirschenhofer@unileoben.ac.at,
Joerg.Thuswaldner@unileoben.ac.at}

\thanks{Both authors are supported by the Franco-Austrian research project I1136 ``Fractals and Numeration'' granted by the French National Research Agency (ANR) and the Austrian Science Fund (FWF), and by the Doctoral Program  W1230 ``Discrete Mathematics'' supported by the FWF. The second author is supported by the FWF project P27050.}

\date{\today}



\begin{document}

\begin{abstract}
For $d \in \mathbb{N}$ the well-known Schur-Cohn region
$\mathcal{E}_d$ consists of all $d$-dimensional vectors
$(a_1,\ldots,a_d)\in\R^d$ corresponding to monic polynomials
$X^d+a_1X^{d-1}+\cdots+a_{d-1}X+a_d$ whose roots all lie in the open
unit disk. This region has been extensively studied over decades.
Recently, Akiyama and Peth\H{o} considered the subsets
$\mathcal{E}_d^{(s)}$ of the Schur-Cohn region that correspond to
polynomials of degree $d$ with exactly $s$ pairs of nonreal roots.
They were especially interested in the $d$-dimensional Lebesgue
measures $v_d^{(s)}:=\lambda_d(\mathcal{E}_d^{(s)})$ of these sets
and their arithmetic properties, and gave some fundamental results.
Moreover, they posed two conjectures that we prove in the present
paper. Namely, we show that in the totally complex case $d=2s$ the
formula
\[
\frac{v_{2s}^{(s)}}{v_{2s}^{(0)}} = 2^{2s(s-1)}\binom {2s}s
\]
holds for all $s\in\N$ and in the general case the quotient
$v_d^{(s)}/v_d^{(0)}$ is an integer for all choices $d\in
\mathbb{N}$ and $s\le d/2$. We even go beyond that and prove
explicit formul\ae{} for 
$v_d^{(s)} / v_d^{(0)}$ for
arbitrary $d\in \mathbb{N}$, $s\le d/2$. The ingredients of our
proofs comprise Selberg type integrals, determinants like the Cauchy
double alternant, and partial Hilbert matrices.
\end{abstract}

\maketitle \setcounter{section}{0}

\section{Introduction}
Almost 100 years ago Issai Schur~\cite{Schur:17,Schur:18} introduced and studied what is now called the \emph{Schur-Cohn region}. For each dimension $d\in \mathbb{N}$ this region is defined as the set of all coefficient vectors of monic polynomials of degree $d$ each of whose roots lies in the open unit disk. As we will deal with polynomials with real coefficients in the present paper we give the formal definition for this setting. The $d$-dimensional Schur-Cohn region is defined by
\[
\mathcal{E}_d :=\{ (a_1,\ldots,a_d)\in \mathbb{R}^d \;:\;
\hbox{each root }\xi\hbox{ of } X^d+a_1X^{d-1}+\cdots+a_{d-1}X+a_d \hbox{ satifies }|\xi|<1
\}.
\]
Sometimes polynomials having all roots in the open unit disk are
called  \emph{contractive polynomials}. In the years after Schur's
papers were published, Cohn~\cite{Cohn:22} developed an algorithm to
check if a polynomial is contractive. In
Rahman and Schmeisser~\cite[Section~11.5]{RS:02} further developments and
properties of the Schur-Cohn region are surveyed.

Starting from the 1970s contractive polynomials were studied by Fam
and Meditch~\cite{Fam-Meditch:78} and their co-authors from the
viewpoint of discrete time systems design. In this context
Fam~\cite{Fam:89} calculated the $d$-dimensional Lebesgue measure
$\lambda_d(\mathcal{E}_d)$ of $\mathcal{E}_d$. In particular, he
proved that (here we give a simplified version of \cite[equations~(2.7) and
(2.8)]{Fam:89}, see also~\cite[equation~(9)]{Dub:16})
\begin{equation}\label{eq:full}
v_d := \lambda_d(\mathcal{E}_d) = 2^d \prod_{j=1}^{\lfloor d/2
\rfloor} \left(1+\frac{1}{2j}\right)^{2j-d}.
\end{equation}
More recently, the Schur-Cohn region occurred in connection
with number systems and so-called \emph{shift radix systems} (see
Akiyama {\it et
al.}~\cite{Akiyama-Borbeli-Brunotte-Pethoe-Thuswaldner:05} or
Kirschenhofer and Thuswaldner~\cite{KT:14}).

Given a polynomial $P\in\mathbb{R}[X]$ of degree $d$ with $r$ real
and $2s$ nonreal roots we call $(r,s)$ the \emph{signature} of $P$.
Since $d=r+2s$ for given degree $d$ the signature is determined by
the value $s$ of pairs of conjugate nonreal roots. In ~\cite{API},
Akiyama and Peth\H{o} studied a partition of $\mathcal{E}_d$ in
terms of the signature of contractive polynomials (see
also~\cite{APII}). In particular, they were interested in the sets
\[
\mathcal{E}_d^{(s)} := \{(a_1,\ldots,a_d)\in \mathcal{E}_d
\;:\;
X^d+a_1X^{d-1}+\cdots+a_{d-1}X+a_d \hbox{ has exactly } s \hbox{ pairs of nonreal roots}
\}
\]
and their volumes
\begin{equation}\label{eq:quotvol}
v_d^{(s)} := \lambda_d(\mathcal{E}_d^{(s)}).
\end{equation}
In the quest for formul\ae{} for these volumes they unveiled
interesting  relations of $v_d^{(s)}$ with Selberg integrals
(introduced in \cite{Selberg:44}) and their generalizations by
Aomoto~\cite{Aomoto:87}. In particular, they expressed $v_d^{(0)}$ in
terms of a classical Selberg integral which led to the simple
formula (see~\cite[Theorem~4.1 and Lemma~5.1]{API})
\[
v_d^{(0)} = 2^{d(d+1)/2}\prod_{j=1}^d\frac{(j-1)!^2}{(2j-1)!}.
\]
Besides that, the studies by Akiyama and Peth\H{o} in \cite{API}
were devoted to arithmetical properties of the quantities
$v_d^{(s)}$. It was proved that these numbers are rational for all
choices of $d$ and $s$. Moreover, certain quotients of these numbers
were studied. In particular, it was shown that $v_d / v_d^{(0)}$ is
an odd integer for all $d\ge 1$ and, based on numerical experiments,
the following conjecture was stated.

\begin{conjecture}[{see \cite[Conjecture~5.1]{API}}]\label{con:integer}
The quotient $v_d^{(s)}/v_d^{(0)}$ is an integer for all $s\le d/2$.
\end{conjecture}

Moreover, for the particular instance $d=2s$
the following surprisingly simple explicit formula was
conjectured.

\begin{conjecture}[{see \cite[Conjecture~5.2]{API}}]\label{con:complex}
For all $s\in \N$ we have
\begin{equation}\label{eq:cxc}
\frac{v_{2s}^{(s)}}{v_{2s}^{(0)}} = 2^{2s(s-1)}\binom {2s}s.
\end{equation}
\end{conjecture}

In the present paper we confirm both these conjectures (see
Theorem~\ref{Integralratios} and Theorem~\ref{totalcomplex} below).
We note that a special instance of the first conjecture was
confirmed recently by Kirschenhofer and Weitzer~\cite{KiWei}, where
the formula
\begin{equation}\label{eq:KWLegendre}
\frac{v_d^{(1)}}{v_d^{(0)}} = \frac{P_d(3)-2d-1}{4}
\end{equation}
was established. Here $P_d$ is the $d$-th Legendre polynomial.
Integrality of this quotient was then derived by standard properties
of Legendre polynomials.

In our proof of the conjectures of Akiyama and Peth\H{o} various
ingredients are needed. We start with a theorem established by these
authors ({\it viz.}~\cite[Theorem~2.1]{API}) to write $v_d^{(s)}$ as
a Selberg-type integral. It should be noted that many questions and problems in number theory, combinatorics, CFT, and other mathematical areas are related to integrals of this type (compare {\it e.g.}~\cite{AAR:99, FW:08}).
For the treatment of the Selberg type integrals occurring in our proof in the context of random Vandermonde matrices an approach proposed in Alastuey and Jancovici~\cite{AlastueyJankovici:xx} turns out to be very valuable. Besides that, methods from ``advanced determinant calculus'' (see~\cite{Krattenthaler2001}) are used to obtain closed forms for variants of the Cauchy double alternant that come up along the way. Moreover,
we need minors of the Hilbert matrix, which is known in numerics for
its numerical instability (see {\it e.g.}~\cite{Todd:60}), to further
simplify our formul\ae{} in order to finally establish the
integrality of the quotients under scrutiny.

The paper is organized as follows. In Section~\ref{sec:cx} we
confirm~Conjecture~\ref{con:complex} of Akiyama and Peth\H{o}. In
particular, we prove \eqref{eq:cxc} (see Theorem~\ref{totalcomplex})
which immediately implies that this quotient is integral. Although
the proof of this result is much shorter than the proof of the
general case it contains several of the main ideas.
Section~\ref{sec:prep} is devoted to preparatory definitions and
results. It contains the proof of a closed form for a certain convolution of a ``determinant
like'' sum that will be needed in the proof of our main result. Section~\ref{sec:vol} contains a
formula for the volumes $v_d^{(s)}$ (see Theorem~\ref{Detmix}). In
Section~\ref{sec:quo} we prove different formul\ae{} for the quotients
$v_d^{(s)}/v_d^{(0)}$ and establish their integrality (see Theorem~\ref{mixedratio}). This proves Conjecture~\ref{con:integer} of Akiyama and Peth\H{o}. We then discuss how
the known special cases proved by Akiyama and Peth\H{o}~\cite{API}
as well as Kirschenhofer and Weitzer~\cite{KiWei} follow. 

\section{The totally complex case}\label{sec:cx}

To prove the integrality of $v_d^{(s)}/v_d^{(0)}$ turns out to be
rather involved. For this reason  we decided to start this paper
with the proof of Conjecture~\ref{con:complex} concerning the
quotient ${v_{2s}^{(s)}}/{v_{2s}^{(0)}}$. The proof of this
conjecture is easier than the proof of Conjecture~\ref{con:integer}
but already contains some of the ideas needed for the general case
and serves as a ``road map'' for the treatment of the
contribution of the nonreal roots in the general proof. Moreover, we
think that the nice formula for this quotient deserves to be
announced in a separate theorem.

\begin{theorem} \label{totalcomplex} The quotient ${v_{2s}^{(s)}}/{v_{2s}^{(0)}}$  is an
integer for each $s\geq1$. In particular,
\begin{equation*}
\frac{v_{2s}^{(s)}}{v_{2s}^{(0)}} = 2^{2s(s-1)}\binom {2s}s.
\end{equation*}
\end{theorem}

By \cite[Theorem 4.1 and Corollary 5.1]{API} we know that the denominators satisfy
\begin{equation}\label{vreell}
v_{2s}^{(0)}= \frac {2^{s(2s+1)}}{(2s)!}/\prod_{j=0}^{2s-1}\binom
{2j+1}j.
\end{equation}
Therefore, the assertion of the theorem is equivalent to
\begin{equation}\label{vtotalkomplex}
v_{2s}^{(s)}= \frac {2^{s(4s-1)}}{(s!)^2}/\prod_{j=0}^{2s-1}\binom
{2j+1}j.
\end{equation}

As a main step for the proof of Theorem \ref{totalcomplex} we establish the following result.
\begin{lemma} \label{totalcomplexlemma}
For $s\in\mathbb{N}$ we have
\begin{equation*}
v_{2s}^{(s)} =  2^{3s}D_s,
\end{equation*}
where
\begin{equation*}
D_s = \det \left(\frac 1{(2j)^2-(2k-1)^2}\right)_{1\le j,k\le s}.
\end{equation*}
\end{lemma}
\begin{proof}
We start from \cite[Theorem 2.1]{API} for the case of $r=0$ real
roots and $s$ pairs of complex conjugate roots
\begin{equation}\label{eq:not21}
z_j = x_j + iy_j \quad\hbox{and}\quad
\overline{z_j} = x_j - iy_j
 \qquad (1\le j\le s),
\end{equation}
where we use the notation $(X-z_j)(X-\overline{z_j})=X^2-\alpha_j X
+ \beta_j,$ so that
\begin{equation}\label{eq:not22}
\alpha_j = 2x_j \quad\hbox{and}\quad
\beta_j =x_j^2 + y_j^2 \qquad (1\le j\le s).
\end{equation}
With this notation \cite[Theorem 2.1]{API} states that
\begin{equation*}
\begin{split}
v_{2s}^{(s)}= \frac 1 {s!}\int_{D_{0,s}}\prod_{1\le j<k\le s}
(z_j-z_k)(z_j-\overline{z_k})(\overline{z_j}-z_k)(\overline{z_j}-\overline{z_k})\:d\alpha_1\cdots\:d\alpha_s\:d\beta_1\cdots
\:d\beta_s,
\end{split}
\end{equation*}
where $D_{0,s}$ denotes the region $0\le \beta_j\le 1$,
$-2\sqrt{\beta_j}\le \alpha_j\le 2\sqrt{\beta_j}$ for all $1\le j\le
s$.
Since
\begin{equation*}
\begin{split}
\det\frac{\partial{(\alpha_j,\beta_j)}}{\partial{(x_j,y_j)}} = 4y_j
= 2i (\overline{z_j}-z_j),
\end{split}
\end{equation*}
we get

\begin{equation}\label{eq:vorDefDet}
\begin{split}
v_{2s}^{(s)}= \frac {2^s i^s} {s!}\int\limits_{\substack{(x_j,y_j)\in B^+ \\ 1\leq j\leq s}}&\left(\prod_{1\le j<k\le s} (z_k-z_j)(z_k-\overline{z_j})(\overline{z_k}-z_j)(\overline{z_k}-\overline{z_j})\right)\\
      &\cdot \prod_{j=1}^s(\overline{z_j}-z_j)\:d x_1\cdots\:d x_s\:d y_1\cdots \:d y_s
\end{split}
\end{equation}
with $B^+$ denoting the upper half of the unit disk. In the
following step of computation we adopt an idea used by Physicists in
the computation of certain Coulomb potentials (see {\it
e.g.}~\cite{AlastueyJankovici:xx}). Observe that the integrand in
the last integral equals the Vandermonde determinant
$V(z_1,\overline{z_1}, \ldots, z_s, \overline{z_s})$ of order $2s.$
In the following, we denote by $\mathfrak{S}_{m}$ the symmetric group on $m$ letters. 
Inserting the definition of the determinant in \eqref{eq:vorDefDet} we obtain with $\sigma_k := \sigma (k)$ that
\begin{equation}\label{eq:postvdm}
\begin{split}
v_{2s}^{(s)}= \frac {2^s i^s} {s!}\int\limits_{\substack{(x_j,y_j)\in B^+ \\ 1\leq j\leq s}}&\left( \sum_{\sigma\in \mathfrak{S}_{2s}} \sgn{\sigma}
z_1^{\sigma_1-1}\overline{z_1}^{\sigma_2 -1}\cdots
z_s^{\sigma_{2s-1}-1}\overline{z_s}^{\sigma_{2s}-1}\right)
      \:d x_1\:d y_1\cdots \:d x_s\:d y_s.
\end{split}
\end{equation}
From this, transforming to polar coordinates $x_j + i y_j = r_j \exp(i\theta_j)$, we get
\begin{equation}\label{eq:2odd1}
\begin{split}
v_{2s}^{(s)}= \frac {2^s i^s} {s!}&\sum_{\sigma\in
\mathfrak{S}_{2s}} \sgn{\sigma} \prod_{j=1}^s\int_{r_j=0}^1
\int_{\theta_j=0}^\pi r_j^{\sigma_{2j-1}+\sigma_{2j}-1}\exp
\big(i(\sigma_{2j-1}-\sigma_{2j})\theta_j\big)\:d \theta_j\:d r_j.
\end{split}
\end{equation}
For $f:\mathbb{R}\to\mathbb{C}$, we set  $\int f(t)\:d t = \int \operatorname{Re}( f(t))\:d t + i \int  \operatorname{Im}( f(t))\:d t$. Then $\int \exp(i\alpha t)\:d t = \frac{\exp (i\alpha t)}{i\alpha}$ for $\alpha\neq0$,
and we have
\begin{equation}\label{eq:2odd2}
\begin{split}
v_{2s}^{(s)}= \frac {2^s} {s!}&\sum_{\sigma\in \mathfrak{S}_{2s}}
\sgn{\sigma} \prod_{j=1}^s \frac
1{\sigma_{2j-1}+\sigma_{2j}}\frac{(-1)^{\sigma_{2j-1}-\sigma_{2j}}-1}{\sigma_{2j-1}-\sigma_{2j}}.
\end{split}
\end{equation}
Observe that the numerators in the last fraction are equal to $0$ if
$\sigma_{2j-1}-\sigma_{2j}$ is even and equal to $-2$ if
$\sigma_{2j-1}-\sigma_{2j}$ is odd. Therefore, only permutations
$\sigma\in \mathfrak{S}_{2s}$, for which $\sigma_{2j-1}-\sigma_{2j}$
is odd for all $1\le j \le s$ give a nonzero contribution to the
sum. A permutation of this type has the form
\begin{equation*}
\sigma = (\varphi, \psi) = \begin{pmatrix} 1 & 2 & 3 & 4 & \dots & 2j-1 & 2j & \dots & 2s-1 & 2s\\
          \varphi(1) & \psi(2) & \varphi(3) & \psi(4) & \dots & \varphi({2j-1}) & \psi({2j}) & \dots & \varphi({2s-1}) & \psi({2s})
          \end{pmatrix},
\end{equation*}
where $\varphi$ is a permutation in $\mathfrak{S}_{2s}^{o}$, the set of all permutations of $\{1,3,\ldots, 2s-1\}$, and  $\psi$ is a permutation in $\mathfrak{S}_{2s}^{e}$, the set of all prmutations of $\{2,4,\ldots, 2s\}$, or has a form where any of the pairs $(\varphi({2j-1}), \psi({2j}))$ in the second line are interchanged. Each interchange of the latter type
will change the sign of the permutation by the factor $-1$.
Therefore, setting $\varphi_j:=\varphi({2j-1}), \psi_j:=
\psi({2j})$, we get
\begin{equation}\label{eq:2odd3}
\begin{split}
v_{2s}^{(s)}= \frac {2^s} {s!}&
\sum_{\begin{subarray}{c}(\varphi, \psi)\\ \varphi\in
\mathfrak{S}_{2s}^{o}, \psi \in \mathfrak{S}_{2s}^{e}\end{subarray}}
 \text{sgn}
(\varphi, \psi) \prod_{j=1}^s \frac 1{\varphi_j
+\psi_j}\left(\frac{(-1)^{\varphi_j -\psi_j}-1}{\varphi_j
-\psi_j}-\frac{(-1)^{\psi_j -\varphi_j}-1}{\psi_j
-\varphi_j}\right).
\end{split}
\end{equation}
Since $\varphi_j -\psi_j$ is odd, the term in the round brackets
equals $\frac 4{\varphi_j -\psi_j},$ which yields
\begin{equation*}
\begin{split}
v_{2s}^{(s)}= \frac {2^s} {s!}2^{2s}\sum_{(\varphi, \psi)}
\text{sgn} (\varphi, \psi) \prod_{j=1}^s \frac 1{\psi_j^2
-\varphi_j^2}.
\end{split}
\end{equation*}
Observe that
\begin{equation*}
\sigma = (\varphi, \psi) = \begin{pmatrix} 1 & \dots & 2j-1  & \dots & 2s-1 \\
          \varphi_1  & \dots & \varphi_j & \dots & \varphi_s
          \end{pmatrix}  \circ \begin{pmatrix} 2 & \dots & 2j  & \dots & 2s \\
          \psi_1  & \dots & \psi_j & \dots & \psi_s
          \end{pmatrix} = \varphi \circ \psi,
\end{equation*}
if  $\varphi$ and $\psi$ are regarded as elements of
$\mathfrak{S}_{2s}$. Therefore, $\text{sgn} (\varphi, \psi) =
\text{sgn} (\varphi) \text{sgn} (\psi)$ and we get
\begin{equation*}
\begin{split}
v_{2s}^{(s)}= \frac {2^{3s}} {s!}\sum_{\varphi\in
\mathfrak{S}_{2s}^{o}} \text{sgn} (\varphi) \sum_{\psi\in
\mathfrak{S}_{2s}^{e}} \text{sgn} (\psi)\prod_{j=1}^s \frac
1{\psi_j^2 -\varphi_j^2}.
\end{split}
\end{equation*}
By the definition of the determinant we have
\begin{equation*}
\begin{split}
 \sum_{\psi\in \mathfrak{S}_{2s}^{e}} \text{sgn} (\psi)\prod_{j=1}^s \frac 1{\psi_j^2 -\varphi_j^2} = \det \left(\frac 1{(2j)^2-\varphi_k^2}\right)_{1\le j,k\le s}.
\end{split}
\end{equation*}
The determinant on the right hand side needs $\text{sgn} (\varphi)$ interchanges of columns to be transformed to
\begin{equation}\label{Dd}
\begin{split}
  D_s:= \det \left(\frac 1{(2j)^2-(2k-1)^2}\right)_{1\le j,k\le s}.
\end{split}
\end{equation}
Therefore, we have established
\begin{equation}
\begin{split}
 v_{2s}^{(s)}= \frac {2^{3s}} {s!}\sum_{\varphi\in \mathfrak{S}_{2s}^{o}} \text{sgn} (\varphi) \text{sgn} (\varphi) D_s = 2^{3s}D_s,
\end{split}
\end{equation}
which completes the proof of the lemma.
\end{proof}
We finish the proof of Theorem~\ref{totalcomplex} by the following product formula for the determinant $D_s.$
\begin{lemma}
\begin{equation}
D_s = \frac {2^{4s(s-1)}}{(s!)^2}/\prod_{j=0}^{2s-1}\binom {2j+1}j.
\end{equation}
\end{lemma}
\begin{proof}
We start with the observation that $D_s$ is of the form
\begin{equation*}
\begin{split}
 D_s = \det \left(\frac 1{X_j+Y_k}\right)_{1\le j,k\le s}, \text{ with } X_j=(2j)^2 \text{ and } Y_k=-(2k-1)^2.
\end{split}
\end{equation*}
Using the Cauchy double alternant formula ({\it
cf.~e.g.}~\cite{Krattenthaler2001})
\begin{equation}\label{doublealternant}
\det \left(\frac 1{X_j+Y_k}\right)_{1\le j,k\le s}= \frac
{\prod\limits_{1\le j<k\le s}(X_j-X_k)(Y_j-Y_k)}{\prod\limits_{1\le
j,k \le s}(X_j+Y_k)}
\end{equation}
we find
\begin{equation*}
\begin{split}
 D_s = (-1)^{s(s-1)/2} 2^{2s(s-1)}\frac {\prod\limits_{1\le j<k\le s}(k-j)^2(k+j)(k+j-1))}{\prod\limits_{1\le j,k \le s}(2j-2k+1)(2j+2k-1)}.
\end{split}
\end{equation*}
Therefore,
\begin{equation}\label{Ddquotients}
\frac {D_{s+1}}{D_s} = 
\frac {2^{8s}}{(s+1)^2\binom{4s+3}{2s+1}\binom {4s+1}{2s}}.
\end{equation}
From the last equality the lemma follows immediately by induction.
\end{proof}
Thus the proof of Theorem \ref{totalcomplex} is complete.

\section{Preparations for the general case}\label{sec:prep}

This section contains some definitions and auxiliary results that will be needed in the proof of our main result.

Throughout the paper finite (totally) ordered sets will play a prominent role.
If $X$ and $Y$ are finite ordered sets we write $X\subseteq Y$ to state
that $X$ is a (possibly empty) ordered subset of $Y$ whose elements inherit the order from $Y$. Moreover $Y\setminus X$ denotes the finite ordered subset of $Y$ containing all elements of $Y$ that are not contained in $X$. Also other set theoretic notions will be carried over to finite ordered sets in the same vein. To indicate the order we sometimes write $\{X_1<\cdots < X_r\}$ instead of $\{X_1,\ldots, X_r\}$ for a finite ordered set. We also use the following notation. Given an ordered subset $I=\{i_1<\cdots < i_m\}$ of $\{1<\cdots< r\}$ we denote $\overline{I} = \{1<\cdots< r\} \setminus I$ and we write $X_I=\{X_{i_1} < X_{i_2} < \cdots < X_{i_m}\}$. 

For $r\in\mathbb{N}$ let $X=\{X_1 <\cdots < X_r\}$ be an ordered set of indeterminates and consider sums of the form
\begin{equation}\label{eq:newHm}
H(X)=H_r(X)=
\sum_{\sigma \in \mathfrak{S}_r} \text{sgn}(\sigma) \prod_{i=1}^r\frac{1-X_i}{1-X_{\sigma(1)}\cdots X_{\sigma(i)}}.
\end{equation}
We now give a closed form for convolutions of these sums which will immediately imply an identity that is needed in the proof of our main results. 

\begin{lemma}\label{lem:newconv}
Let $[r]=\{1<\cdots<r\}$. Then, setting $X=\{X_{1}  < \cdots < X_{r}\}$, the sums $H$ from \eqref{eq:newHm} admit the convolution formula
\begin{equation}\label{eq:nevConv}
\begin{split}
\sum_{K\subseteq [r]} H(X_K)H(X_{\overline{K}})= 2^{\lceil r/2 \rceil}
\prod_{\begin{subarray}{c}j=1 \\ j+r {\small \rm \; odd} \end{subarray}}^{r} (1+X_j)
\prod_{\begin{subarray}{c}j=1 \\ j {\small \rm \; even} \end{subarray}}^{r} (1-X_j)
\prod_{\begin{subarray}{c}1\le j < k \le r \\ j-k{\small \rm \; even} \end{subarray}} (X_j-X_k)
\prod_{\begin{subarray}{c}1\le j < k \le r \\ j-k{\small \rm \; odd} \end{subarray}} \frac{1}{1-X_jX_k}.
\end{split}
\end{equation}
Here the sum runs over all ordered subsets $K$ of $[r]$.
\end{lemma}

\begin{proof}
First we note that, according to \cite[${\rm 3^{rd}}$ identity on p.\ 6812]{KLW:15}, for a finite ordered set $X$ the sum $H(X)$ defined in \eqref{eq:newHm} admits the closed form
\begin{equation}\label{eq:Hclosed}
H(X) =  \prod_{1\le j<k\le r}\frac{X_j-X_k}{1-X_jX_k}.
\end{equation}
Indeed, \eqref{eq:Hclosed} is proved by induction using a recurrence relation which is established by splitting the sum in \eqref{eq:newHm} w.r.t.\ the value of $\sigma(m)$  (see \cite[proof of equation~(4.9)]{KLW:15}).
This closed form will be used throughout the proof.  

Let $f_r(X_1,\ldots, X_r)$ and $g_r(X_1,\ldots, X_r)$ denote the left and right hand side of \eqref{eq:nevConv}, respectively. To prove the lemma it suffices to show that $f_r$ as well as $g_r$ satisfy the recurrence
\begin{equation}\label{eq:rech}
\begin{split}
h_r(X_1,\ldots,X_r)
\,=\,&
\frac {(1+X_{r-1})(1-X_{r})}
      {1-X_{r-1}X_r}
      h_{r-1}(X_1,\ldots,X_{r-1})\\
&-\sum_{\begin{subarray}{c}i=1\\ i+r \,\rm{even} \end{subarray}}^{r-1}
  \frac{(1+X_{r-1})(1-X_i)}
       {1-X_{r-1}X_i}
h_{r-1}(X_1,\ldots,X_{i-1},
     X_r,X_{i+1},\ldots,X_{r-1})
\end{split}
\end{equation}
for $r\ge 2$ with $h_1(X_1)=2$. This can be proved using the partial fraction technique.

In order to establish the recurrence \eqref{eq:rech} for $f_r$ we first observe that, since the summands on the left hand side of \eqref{eq:nevConv} are invariant under the involution $K\mapsto \overline{K}$,
\begin{equation}\label{eq:barring}
\sum_{K\subseteq [r]} H(X_K)H(X_{\overline{K}})=2 \sum_{\begin{subarray}{c} K\subseteq [r] \\ r-1 \in K \end{subarray}} H(X_K)H(X_{\overline{K}})=2 \sum_{\begin{subarray}{c} K\subseteq [r] \\ r-1 \not\in K \end{subarray}} H(X_K)H(X_{\overline{K}}).
\end{equation}
Now set
\begin{align*}
p_{r+1}(X_1,\ldots,X_{r+1}) =&  \sum_{\begin{subarray}{c} K\subseteq [r] \\ r-1 \in K \end{subarray}} H(X_K)H(X_{\overline{K}}) \prod_{i\in K}\frac{X_i-X_{r+1}}{1-X_iX_{r+1}},  \\ 
q_{r+1}(X_1,\ldots,X_{r+1}) =& \sum_{\begin{subarray}{c} K\subseteq [r] \\ r-1 \not\in K \end{subarray}} H(X_K)H(X_{\overline{K}}) \prod_{i\in K}\frac{X_i-X_{r+1}}{1-X_iX_{r+1}}. 
\end{align*}
We evaluate $p_{r+1}(X_1,\ldots,X_{r}, -1)$ in two ways. Firstly, we set $X_{r+1}=-1$ directly in the definition of $p_{r+1}(X_1,\ldots,X_{r+1})$ and use \eqref{eq:barring} to replace the condition $r-1\in K$ by a factor $\frac12$ in the resulting expression. Secondly, we insert the partial fraction expansion
\[
\begin{split}
\frac{1}{1-X_{r-1}X_{r+1}}\prod_{\begin{subarray}{c} i\in K \\ i\neq r-1 \end{subarray}}\frac{X_i-X_{r+1}}{1-X_iX_{r+1}}
=&
\sum_{\begin{subarray}{c} i\in K \\ i\not=r-1 \end{subarray}}\frac{1-X_i^2}{(X_{r-1}-X_i)(1-X_iX_{r+1})}
\prod_{ j\in K\setminus\{i,r-1\}}\frac{1-X_jX_i}{X_j-X_i} \\
&+ \frac{1}{1-X_{r-1}X_{r+1}}\prod_{ j\in K\setminus\{r-1\}}\frac{1-X_jX_{r-1}}{X_j-X_{r-1}}
\end{split}
\]
in the product containing the indeterminate $X_{r+1}$ in $p_{r+1}$, then set $X_{r+1}=-1$, interchange the sums, and reorder the indeterminates appropriately. This yields the identity ($\delta_{r,J}$ is $1$ if $r\in J$ and $0$ otherwise)
\begin{equation}\label{eq:rec1}
\begin{split}
\frac12f_{r} (X_1,\ldots, X_r)=& \, p_{r+1}(X_1,\ldots,X_{r}, -1) \\
&\hskip -1cm = -(1+X_{r-1})\sum_{i =1}^{r-2}
\frac{1-X_i}{1-X_{i}X_{r-1}}
\sum_{\begin{subarray}{c} J \subseteq \{1<\cdots < i-1 < r-1 < i+1 < \cdots < r-2 < r\} \\ r-1 \in J \end{subarray}}
\hskip -1.7cm H(X_J)H(X_{\overline{J}})(-1)^{\delta_{r,J}}\\
& + \sum_{J \subseteq \{1 < \cdots < r-2 < r\}}H(X_J)H(X_{\overline{J}})(-1)^{\delta_{r,J}}\\
&+\frac{(1+X_{r-1})(1-X_r)}{1-X_{r}X_{r-1}}
\sum_{\begin{subarray}{c} J \subseteq \{1 < \cdots < r-2 < r-1\} \\ r-1 \in J \end{subarray}}
H(X_J)H(X_{\overline{J}})
\end{split}
\end{equation}
(note that the sum in the second line vanishes by \eqref{eq:barring}). 

Evaluating the sum $q_{r+1}(X_1,\ldots,X_{r+1})$ in two ways using the partial fraction expansion w.r.t.\ $X_{r+1}$ of 
\[
\frac{1}{X_{r-1}-X_{r+1}}\prod_{\begin{subarray}{c} i\in K 
\end{subarray}}\frac{X_i-X_{r+1}}{1-X_iX_{r+1}},
\]
(in this case $r-1\not\in K$ holds), by similar reasoning as above we gain the formula
\begin{equation}\label{eq:rec2}
\begin{split}
\frac12 f_{r} (X_1,\ldots, X_r)=&\, q_{r+1}(X_1,\ldots,X_{r}, -1) \\
& \hskip -1.5cm =-(1+X_{r-1})\sum_{i =1}^{r-2}
\frac{1-X_i}{1-X_{i}X_{r-1}}
\sum_{\begin{subarray}{c} J \subseteq \{1<\cdots < i-1 < r-1 < i+1 < \cdots < r-2 < r\} \\ r-1 \in J \end{subarray}}
\hskip -1.7cm H(X_J)H(X_{\overline{J}})(-1)^{r-i+\delta_{r,J}}
\\
&
+\frac{(1+X_{r-1})(1-X_r)}{1-X_{r}X_{r-1}}
\sum_{\begin{subarray}{c} J \subseteq \{1 < \cdots < r-2 < r-1\} \\ r-1 \in J \end{subarray}}
H(X_J)H(X_{\overline{J}}).
\end{split}
\end{equation}

Adding \eqref{eq:rec1} and \eqref{eq:rec2} the summands corresponding to indices $i$ with $r-i$ odd vanish. In the remaining summands interchanging the indeterminates $X_{r}$ and $X_{r+1}$ absorbs the sign $(-1)^{\delta_{r,J}}$ and, hence, \eqref{eq:rech} holds for $f_r$.

To show that $g_r$ also satisfies \eqref{eq:rech} we start with $g_{r+1}(X_1,\ldots, X_{r+1})$. Indeed, to get the left hand side of \eqref{eq:rech} just set $X_{r+1}=-1$. For the right hand side expand $g_{r+1}(X_1,\ldots, X_{r+1})/(X_{r+1}-X_{r-1})$ in partial fractions w.r.t.\  $X_{r+1}$, multiply by $X_{r+1}-X_{r-1}$, and set $X_{r+1}=-1$ again. 
\end{proof}

Let $r\in\mathbb{N}$ and let $Y=\{Y_{1}  < \cdots < Y_{r}\}$ be an ordered set of indeterminates. Later we will need sums of the form
\begin{equation}\label{newSm}
S(Y)=S_r(Y):= \sum_{\sigma\in \mathfrak{S}_r} \sgn{\sigma} \frac
1{Y_{\sigma(1)}(Y_{\sigma(1)}+Y_{\sigma(2)})\cdots
(Y_{\sigma(1)}+\cdots+Y_{\sigma(r)})}.
\end{equation}

Substituting $X_i=q^{Y_i}$
and letting $q$ tend to $1$ it  now follows from \eqref{eq:Hclosed} that
\begin{equation}\label{eq:SRNEW}
S_r(Y) = \frac{1}{Y_1\cdots Y_n} \prod_{1\le j  < k \le r} \frac{Y_k-Y_j}{Y_k+Y_j}.
\end{equation}

By the same operations, the following corollary, which will be used later on, is an immediate consequence of  \eqref{eq:SRNEW} and Lemma~\ref{lem:newconv}.


\begin{corollary}\label{ConvolutionSdlemma}
Let $[r]=\{1<\cdots<r\}$. Then, setting $Y=\{Y_{1}  < \cdots < Y_{r}\}$, the sums $S_m$ from \eqref{newSm} admit the convolution formula
\begin{equation*}
\sum_{K \subseteq [r]} S_{|K|}(Y_K)S_{r-|K|}(Y_{\overline{K}}) = 2^r
\prod_{\begin{subarray}{c}1\le j < k \le r \\ j-k{\small \rm \; even}\end{subarray}} (Y_{k}-Y_{j})
\prod_{\begin{subarray}{c}j=1 \\ j{\small \rm \; odd}\end{subarray}}^r\frac1{Y_j}
\prod_{\begin{subarray}{c}1\le j < k \le r \\ j-k{\small \rm \; odd}\end{subarray}} \frac{1}{Y_{k}+Y_{j}}.
\end{equation*}
Here the sum runs over all ordered subsets $K$ of $[r]$.
\end{corollary}

We now define some notions related to the parity of the elements of a finite ordered set.

\begin{definition}\label{def:oe}
A finite ordered set $X$ consisting of $\lceil |X|/2\rceil$ odd and $\lfloor |X|/2\rfloor$ even numbers  is said to be in {\em increasing oe-order}, if it is of the shape odd, even, odd, even, ..., where the odd and the even ordered subset are both strictly increasing. We denote the sign of the permutation that brings the elements of $X$ from their natural order to the increasing oe-order by $\oesgn{X}$.

Let $X$ be a finite ordered set of strictly ascending positive integers. Define the finite ordered sets  $X_1$ and $X_2$ such that $2X_1$ and $2X_2-1$ constitute the even and odd elements of the finite ordered set $X$ in ascending order, respectively. Then $(X_1,X_2)$ is called the \emph{parity splitting} of $X$. The finite ordered set $X$  is called \emph{parity balanced} if $|X_1|=|X_2|$, {\em i.e.}, if $X$ contains the same number of odd and even elements.
\end{definition}

These concepts are best illustrated by an example. Let $X=\{1<3<4<5<6\}$ be given. Then the elements of $X$ are ordered increasingly. To get $X$ in increasing oe-order $\{1<4<3<6<5\}$ we have to do two transpositions. Thus this reordering is achieved by an even permutation and, hence, $\oesgn{X} = 1$. The parity splitting of $X$ is $(\{2<3\},\{1<2<3\})$, and since $\{2<3\}$ and $\{1<2<3\}$ have different cardinality we see that $X$ is not parity balanced.

We will need the following property of the oe-order.

\begin{lemma}\label{Vorzeichenlemma}
For $\nu \in\mathbb{N}$ consider the ordered set $X=\{1<2<\cdots< 2\nu\}$. Let $M$ be a parity balanced ordered subset of $X$ and let $N=\overline{M}$. If $(N_1,N_2)$ is the parity splitting of $N$ then
\begin{equation}
\oesgn{M}\oesgn{N} = (-1)^{\sum_{m\in N_1} m + \sum_{n\in N_2} n}.
\end{equation}
\end{lemma}
\begin{proof}
Fix the cardinality $c=|N_1|=|N_2|$ and let $a = \sum_{m\in N_1} m + \sum_{n \in N_2} n$. Note that $a \ge c(c+1)$. We prove the lemma by induction on $a$.

For the induction start observe that if $a=c(c+1)$ then $N_1=N_2=\{1,\ldots,c\}$  and, hence, $\oesgn{M}=\oesgn{N}=1$ (the oe-order and the natural order coincide) and $(-1)^a=1$. Thus the lemma holds in this instance.

For the induction step assume that the lemma is true for all $c(c+1)\le a < a_0$ and consider a constellation $N_1,N_2$ corresponding to the value $a=a_0$. Since $a_0>c(c+1)$ there is $x\in N$ with $x-2\in M$. Now we switch these two elements, {\it i.e.}, we
put $x$ in the set $M$ and $x-2$ in the set $N$. If $x-1 \in N$ then $\oesgn{M}$ remains the same under this change, and $\oesgn{N}$ changes its sign. If $x-1 \in M$ then $\oesgn{N}$ remains the same under this change, and $\oesgn{M}$ changes its sign. Thus in any instance, $\oesgn{M}\oesgn{N}$ changes the sign.

Summing up, in the new constellation we have $a=a_0-1$ and a total sign change of $(-1)$ in $\oesgn{M}\oesgn{N}$ compared to the original constellation. Thus the lemma follows by induction.
\end{proof}

\section{Formul\ae{} for the volumes in the general case}\label{sec:vol}
In this section we treat the case of polynomials having $s$ pairs of
(nonreal) complex conjugate roots and $r$ real roots within the open unit
disk. We will first derive a general formula that expresses the
volume $v_{2s+r}^{(s)}$ of the coefficient space in this instance by means of a sum
of determinants. This forms a generalization of
Lemma~\ref{totalcomplexlemma}.

Using the notations introduced in \eqref{eq:not21} and
\eqref{eq:not22} the result of \cite[Theorem 2.1]{API} states that
the volumes in question admit the integral representation
\begin{equation}
\begin{split}
v_{2s+r}^{(s)}=& \frac 1 {s!r!}\int_{D_{r,s}}\prod_{1\le j<k\le s}
(z_j-z_k)(z_j-\overline{z_k})(\overline{z_j}-z_k)(\overline{z_j}-\overline{z_k})
\\ &
\!\!\cdot\prod_{1 \le j \le s, 1\le k \le r}(z_j-\xi_k)
(\overline{z_j}-\xi_k)\prod_{1\le j<k \le
r}|\xi_k-\xi_r|\:d\alpha_1\cdots\:d\alpha_s\:d\beta_1\cdots
\:d\beta_s\:d\xi_1\cdots\:d\xi_r,
\end{split}
\end{equation}
where $D_{r,s}$ denotes the region given by $0\le \beta_j\le 1,
-2\sqrt{\beta_j}\le \alpha_j\le 2\sqrt{\beta_j}$ for all $1\le j\le
s$ and $-1\le \xi_j\le 1$ for all $1\le j \le r.$ Proceeding as in
\eqref{eq:postvdm} with the ``complex part'' of the integral and
bringing the variables $\xi_1,\ldots,\xi_r$ in increasing order
yields
\begin{equation}\label{eq:1stV}
v_{2s+r}^{(s)}= \frac {2^s i^s} {s!}\int\limits_{\substack{(x_j,y_j)\in B^+ \\ 1\leq j\leq s}}\int_{-1\le\xi_1\le\cdots\le\xi_r\le 1}  V_{2s,r}
      \:d x_1\cdots\:d x_s\:d y_1\cdots \:d y_s\:d\xi_1\cdots\:d\xi_r,
\end{equation}
where
\begin{equation}\label{eq:4vdm}
\begin{split}
V_{2s,r}:=V(z_1,\overline {z_1},\dots,
z_s,\overline{z_s},\xi_1,\dots,\xi_r)
\end{split}
\end{equation}
is the Vandermonde determinant of order $2s+r$. Laplace expansion of
this determinant by the first $2s$  columns (which are the columns
corresponding to the ``complex'' part) yields
\begin{equation}\label{eq:vdmLap}
\begin{split}
V_{2s,r}=\sum_{\begin{subarray}{c}M\subseteq \{1,\dots,2s+r\}\\ |M|= 2s\end{subarray}}
V_{2s,r}^{M,\{1,\dots,2s\}}V_{2s,r}^{N,\{2s+1,\dots,2s+r\}}(-1)^{\sum_{j\in
M} j +\binom {2s+1}2},
\end{split}
\end{equation}
where the sum runs over finite ordered sets $M,$ the ordered set $N$ is defined as the complement $N=\{1,\dots,2s+r\}\setminus M$, and $V_{2s,r}^{I,J}$ is the minor of the Vandermonde matrix \eqref{eq:4vdm} with the rows having index in $I$ and the columns having index in $J$. Inserting
\eqref{eq:vdmLap} in \eqref{eq:1stV} we obtain
\begin{equation}\label{v2s+rs}
\begin{split}
v_{2s+r}^{(s)}= \frac {2^s i^s} {s!}\sum_{M}(-1)^{\sum_{j\in M} j +\binom {2s+1}2}&\int\limits_{\substack{(x_j,y_j)\in B^+ \\ 1\leq j\leq s}}V_{2s,r}^{M,\{1,\dots,2s\}}\:d x_j \:d y_j\\
&\cdot\int_{-1\le\xi_1\le\cdots\le\xi_r\le
1}V_{2s,r}^{N,\{2s+1,\dots,2s+r\}}\:d \xi_1\cdots\:d\xi_r.
\end{split}
\end{equation}
Proceeding with the first integral as in \eqref{eq:2odd1}, \eqref{eq:2odd2}, and \eqref{eq:2odd3} we see that this integral equals zero if $M$ is not parity balanced in the sense of Definition~\ref{def:oe}. Furthermore, letting $(M_1,M_2)$ be the parity splitting of $M$ (see again Definition~\ref{def:oe}), with the same arguments as in the instance $r=0$ in Section~\ref{sec:cx} (for the definition of $\oesgn{M}$ see also Definition~\ref{def:oe}) for parity balanced $M$ the first integral in \eqref{v2s+rs} satisfies 
\begin{equation}
\begin{split}
\frac {2^s i^s} {s!}&\int\limits_{\substack{(x_j,y_j)\in B^+ \\ 1\leq j\leq s}}V_{2s,r}^{M,\{1,\dots,2s\}}\:d x_j \:d y_j = 2^{3s}\oesgn{M}
D_s^{M_1,M_2},
\end{split}
\end{equation}
where $D_s^{M_1,M_2}$ is the minor of the determinant $D_s$ from
\eqref{Dd}, {\it i.e.},
\begin{equation}\label{Ddminors}
\begin{split}
  D_s^{M_1,M_2}:= \det \left(\frac 1{(2j)^2-(2k-1)^2}\right)_{j\in M_1, k\in M_2}.
\end{split}
\end{equation}
Observing that
\begin{equation*}
\begin{split}
(-1)^{\sum_{j\in M} j +\binom {2s+1}2} = (-1)^{|M_2|+(2s+1)s} = 1
\end{split}
\end{equation*}
since $|M_2| = s$, we have
\begin{equation}\label{mixedvolume1}
\begin{split}
v_{2s+r}^{(s)}= 2^{3s}\sum_{\begin{subarray}{c}M\subseteq
\{1,\dots,2s+r\}, |M|= 2s \\ M \hbox{ \small parity
balanced}\end{subarray}}\oesgn{M} D_s^{M_1,M_2}
&\int_{-1\le\xi_1\le\cdots\le\xi_r\le
1}V_{2s,r}^{N,\{2s+1,\dots,2s+r\}}\:d \xi_1\cdots\:d\xi_r,
\end{split}
\end{equation}
where $(M_1,M_2)$ is the parity splitting of $M$, and
$N=\{1,\dots,2s+r\}\setminus M$ (observe that $N$ consists of $\lfloor
r/2\rfloor$ even and $\lceil r/2\rceil $ odd numbers).

Now we turn our attention to the ``real'' integral. For the finite ordered set $N=\{n_1<\cdots<n_r\}$ we have
\begin{equation}\label{Ir1}
\begin{split}
I_r(N): &=\int_{-1\le\xi_1\le\cdots\le\xi_r\le 1}V_{2s,r}^{N,\{2s+1,\dots,2s+r\}}\:d \xi_1\cdots\:d\xi_r \\
&= \int_{-1\le\xi_1\le\cdots\le\xi_r\le 1}V(\xi_1,\dots,\xi_r;N)\:d \xi_1\cdots\:d\xi_r,
\end{split}
\end{equation}
where $V(\xi_1,\dots,\xi_r;N)$ is given by
\begin{equation}\label{PartialVandermonde}
\begin{split}
V(\xi_1,\dots,\xi_r;N):=\det \big(\xi_k^{n_j-1}\big)_{1\le j,k\le r}.
\end{split}
\end{equation}
These determinants are strongly related to the well-known {\em Schur functions}, {\it viz.}, $V(\xi_1,\dots,\xi_r;N)$ is just a product of a Schur function and a Vandermonde determinant  (see~\cite[p.~40, equation~(3.1)]{Macdonald:95}).
Let
\begin{equation}\label{N'}
\begin{split}
N'=\{n_1',\ldots,n_r'\}
\end{split}
\end{equation}
be the ordered set containing the arrangement of the elements of $N$ in increasing oe-order. Then \eqref{mixedvolume1} and \eqref{Ir1} imply the following result.

\begin{lemma}\label{lem:real}
The volume $v_{2s+r}^{(s)}$ satisfies
\begin{equation}\label{real}
\begin{split}
v_{2s+r}^{(s)}= 2^{3s}\sum_{\begin{subarray}{c}M\subseteq
\{1,\dots,2s+r\}, |M|= 2s \\ M \hbox{ \small parity
balanced}\end{subarray}} \oesgn{M}\oesgn{N} D_s^{M_1,M_2}
I_r(N').
\end{split}
\end{equation}
Here $N=\{1,\dots,2s+r\}\setminus M$, and $N'$ is a finite ordered set containing the
elements of $N$ in increasing oe-order (see Definition~\ref{def:oe}
for the terminology).
\end{lemma}

It now remains to compute the integrals
\begin{equation}\label{IR}
\begin{split}
I_r(N') = \sum_{k=0}^r \int_{-1\le\xi_1\le\cdots\le\xi_k\le 0}\int_{0\le\xi_{k+1}\le\dots\le\xi_r\le 1} V(\xi_1,\dots,\xi_r;N')\:d \xi_1\cdots\:d\xi_r.
\end{split}
\end{equation}
By the tight relation of $V(\xi_1,\dots,\xi_r;N')$ to Schur functions mentioned after \eqref{PartialVandermonde}  we see that these integrals resemble the Selberg-type integrals studied in \cite[Section~2]{FW:08}. Indeed,  the {\em Jack polynomials} (also called {\em Jack's symmetric functions}) used there are generalizations of Schur functions (see~\cite[p.~379]{Macdonald:95}). The difference is that in our instance we only have skew symmetric rather than symmetric integrands in \eqref{IR}. Let us fix $k$ for the moment and expand the determinant by its first $k$ columns. Then the integrals in the
$k$th summand of \eqref{IR} can be rewritten as
\begin{equation}\label{realpart}
\begin{split}
\int_{-1\le\xi_1\le\cdots\le\xi_k\le 0}&\int_{0\le\xi_{k+1}\le\cdots\le\xi_r\le 1} V(\xi_1,\dots,\xi_r;N')\:d \xi_1\cdots\:d\xi_r \\
&\!\!\!\!\!\!\!\!\!\!\!\!\!\!\!\!\!\!\!\!\!\!\!\!\!\!\!\!\!=\sum_{K'=\{n_{i_1}',\dots, n_{i_k}'\}\subseteq N'} (-1)^{\sum_{1\le j\le k} i_j +\binom {k+1}2}\int_{-1\le\xi_1\le\cdots\le\xi_k\le 0}V(\xi_1,\dots,\xi_k;K')\:d \xi_1\cdots\:d\xi_k \\
&\cdot \int_{0\le\xi_{k+1}\le\cdots\le\xi_r\le 1}
V(\xi_{k+1},\dots,\xi_r;N'\setminus K')\:d
\xi_{k+1}\cdots\:d\xi_r.
\end{split}
\end{equation}
Substituting $\eta_i=-\xi_i$ for $1\le i\le k$ in the first integral in \eqref{realpart} yields
\begin{equation}\label{eq:iilast}
\begin{split}
&\!\!\!\!\!\!\!\!\!\!\!\!\!\!\!\!\int_{-1\le\xi_1\le\cdots\le\xi_k\le 0}V(\xi_1,\dots,\xi_k;K')\:d \xi_1\cdots\:d\xi_k
\\
 &=\int_{1\ge\eta_1\ge\cdots\ge\eta_k\ge 0}(-1)^{\sum_{1\le j\le k} (n_{i_j}'-1) -k }V(\eta_1,\dots,\eta_k;K')(-1)^k\:d \eta_1\cdots\:d\eta_k\\
  &=\int_{0\le\eta_k\le\cdots\le\eta_1\le 1}(-1)^{k+\sum_{1\le j\le k} (n_{i_j}'-1) +\binom {k}2}V(\eta_1,\dots,\eta_k;K')(-1)^k\:d \eta_1\cdots\:d\eta_k\\
  &=(-1)^{\sum_{1\le j\le k} n_{i_j}'+\binom {k+1}2} J_k(K'),
\end{split}
\end{equation}
where for a finite ordered set $W=\{w_1,\ldots, w_m\}$ we set
\begin{equation}\label{Sm}
\begin{split}
J_m(W):= \int_{0\le u_1\le\cdots\le u_m\le 1} V(u_1,\dots, u_m;W)\:d u_1\cdots\:d u_m.
\end{split}
\end{equation}
Using this notation, inserting \eqref{eq:iilast} into
\eqref{realpart} and then \eqref{realpart} into \eqref{IR}, and
observing that, because $n_1',\ldots, n_r'$ are in increasing
oe-order the numbers $n_{i_j}'+ i_j$ are all even, we finally get
\begin{equation}\label{Ir}
\begin{split}
I_r(N')= \sum_{K' \subseteq N'} J_{|K'|}(K') J_{r-|K'|}(N'\setminus K').
\end{split}
\end{equation}

In the next step we replace the integrals $J_m(W)$ by the sum $S_m(W)$ from \eqref{newSm}. This is justified by the following lemma.

\begin{lemma}\label{lem:JS}
For each $m\in \mathbb{N}$ and each ordered set $W=\{w_1 < \cdots < w_m\}$ of complex numbers with positive real part we have
\[
J_m(W)=S_m(W),
\]
where $S_m(W)$ and $J_m(W)$ are defined in \eqref{newSm} and \eqref{Sm}, respectively.
\end{lemma}

\begin{proof}
Recall that
\begin{equation*}
\begin{split}
J_m(W)= \int_{0\le u_1\le\cdots\le u_m\le 1}
\det(u^{w_j-1}_k)_{1\le j,k\le r }\:d u_1\cdots\:d u_m.
\end{split}
\end{equation*}
With the substitution $u_k:=v_k v_{k+1}\cdots v_m,\ 1\le k\le m,$ the
integral turns to
\begin{align}\label{BoesesIntegral}
\begin{aligned}
J_m&(W)= \int_{0\le v_1,v_2,\dots,v_m\le 1}
\det(v^{w_j-1}_k\cdots v^{w_j-1}_m)_{1\le j,k\le m }v^1_2 v^2_3\cdots
v^{m-1}_m\:d v_1\cdots\:d v_m\\
&=\sum_{\sigma\in \mathfrak{S}_m} \sgn{\sigma} \int_{0\le
v_1,v_2,\dots,v_m\le
1}v_1^{w_{\sigma(1)}-1}v_2^{w_{\sigma(1)}+w_{\sigma(2)}-1}\cdots
v_m^{w_{\sigma(1)}+\cdots+w_{\sigma(m)}-1}\:d v_1\cdots\:d v_m\\
&=\sum_{\sigma\in \mathfrak{S}_m} \sgn{\sigma} \frac
1{w_{\sigma(1)}(w_{\sigma(1)}+w_{\sigma(2)})\cdots
(w_{\sigma(1)}+\cdots+w_{\sigma(m)})}\\
&= S_m(W). \qedhere
\end{aligned}
\end{align}
\end{proof}

For the following considerations we have to make some amendments in case $|N|$ is odd. In particular, we define
\begin{equation}\label{2tequation}
2t= \begin{cases} r &\text{if $r$ is even},\\
    r+1 &\text{if $r$ is odd}
    \end{cases}
\end{equation}
and set
\begin{equation}\label{tildeN1}
\widetilde {N}= \begin{cases} N &\text{if $r$ is even},\\
    N\cup \{2(s+t)\} &\text{if $r$ is odd}.
    \end{cases}
\end{equation}

We now apply Corollary~\ref{ConvolutionSdlemma} in combination with the last lemma to the evaluation of the sums $\sum_{K'} J_{|K'|}(K')J_{r-|K'|}(N'\setminus K')$ occurring in \eqref{Ir}.
\begin{lemma}\label{Reelllemma}
Let $(\widetilde{N}_1,N_2)$ be the parity splitting of $\widetilde N$. Then the ``real'' integrals $I_r(N')$ with $N'$ as in \eqref{N'} satisfy
\begin{equation}\label{eq:JJJ}
\begin{split}
I_r(N')
= 2^r B^{\widetilde{N}_1,N_2},
 \end{split}
\end{equation}
 where $B^{\widetilde{N}_1,N_2}$ is the minor with row indices from $\widetilde{N}_1$ and column indices from  $N_2$ of the matrix
\begin{equation*}
\begin{split}
 B=& \begin{cases}
 \left(\frac 1{(2k-1)(2j+2k-1)}\right)_{1\le j,k\le s+t} &\text{if } r=2t \text{ even},\\
\left(\begin{matrix}\frac 1{(2k-1)(2j+2k-1)},& 1\le j \le s+t-1\\
\frac 1{2k-1},& j=s+t \end{matrix}\right)_{1\le j,k\le s+t}
&\text{if $r=2t-1$ odd.} \end{cases}
\end{split}
\end{equation*}
\end{lemma}
\begin{proof}
To prove the result for the instance $r=2t$ we use Corollary~\ref{ConvolutionSdlemma} with the choice $Y_j=n_j'$, where $n_j'$ is as in  \eqref{N'}. Combining this corollary with Lemma~\ref{lem:JS} we obtain
\begin{equation*}
I_r(N')= 2^r
\prod_{\begin{subarray}{c}1\le j < k \le r \\ j-k{\small \rm \; even}\end{subarray}} (n'_{k}-n'_{j})
\prod_{\begin{subarray}{c}j=1 \\ j{\small \rm \; odd}\end{subarray}}^r\frac1{n'_j}
\prod_{\begin{subarray}{c}1\le j < k \le r \\ j-k{\small \rm \; odd}\end{subarray}} \frac{1}{n'_{k}+n'
_{j}}.
\end{equation*}
Using the double alternant formula \eqref{doublealternant} this yields
\begin{equation*}
I_r(N') = 2^r\det\left(\frac{1}{n'_{2j}-n'_{2k-1}}\right)_{1\le j,k\le t}\;\prod\limits_{\begin{subarray}{c}k=1 \\ k{\small \rm \; odd}\end{subarray}}^r\ \frac1{n'_{k} }
\\
=
2^r\det\left(\frac{1}{n'_{2k-1}(n'_{2j}-n'_{2k-1})}\right)_{1\le j,k\le t}
\end{equation*}
and the result follows because the last determinant equals $B^{\widetilde{N}_1,N_2}$.

In the instance $r=2t-1$ the result follows by the same reasoning. The only difference is that instead of using the double alternant formula \eqref{doublealternant} we have to use \begin{equation}\label{odddoublealternant}
\det \left(\begin{matrix} \frac 1{X_j+Y_k}, &1\le j\le n-1\\
1, &j=n \end{matrix} \right)_{1\le j,k \le n}= \frac {\prod\limits_{1\le
j<k\le n-1}(X_j-X_k)\prod\limits_{1\le j <k\le n}(Y_j-Y_k)}{\prod\limits_{1\le
j\le n-1, 1\le k \le n}(X_j+Y_k)}.
\end{equation}
This identity can in turn be proved along the lines of the usual proof of \eqref{doublealternant}.
\end{proof}

Now we are ready to prove our first main result on the ``mixed'' volumes.

\begin{theorem}\label{Detmix}
Let $v_{d}^{(s)}$ denote the volume of the coefficient space of real
polynomials of degree $d=2s+r$ with $s$ pairs of (nonreal) complex
conjugate roots and $r$ real roots in the open unit disk. Then, for
$s\le n,$
\begin{equation*}
\frac {v_{2n}^{(s)}}{2^{2n+s}}=\sum_{\begin{subarray}{c} J\subseteq
\{1,\dots,n\}\\ |J|=s \end{subarray}}\det \big(c_{j,k}(J)\big)_{1\le j,k\le
n}=\sum_{\begin{subarray}{c} J\subseteq \{1,\dots,n\}\\ |J|=s
\end{subarray}}\det \big({c}'_{j,k}(J)\big)_{1\le j,k\le n}
\end{equation*}
with
\[
c_{j,k}(J)=\begin{cases}\displaystyle \frac{1}{(2j)^2-(2k-1)^2}
& \text{for $j\in J$}, \\[3mm]
\displaystyle \frac{1}{(2k-1)(2j+2k-1)}
& \text{for $j\not\in J$},
\end{cases}
\quad
c'_{j,k}(J)=\begin{cases}\displaystyle \frac{1}{(2j)^2-(2k-1)^2}
& \text{for $k\in J$}, \\[3mm]
\displaystyle \frac{1}{(2k-1)(2j+2k-1)}
& \text{for $k\not\in J$},
\end{cases}
\]
and, for $s\le n-1,$
\begin{equation*}
\frac {v_{2n-1}^{(s)}}{2^{2n+s-1}}=\sum_{\begin{subarray}{c}
J\subseteq \{1,\dots,n-1\}\\ |J|=s \end{subarray}}\det
\big(\tilde c_{j,k}(J)\big)_{1\le j,k\le n}
\end{equation*}
with
\begin{equation*}
\tilde c_{j,k}(J) = \begin{cases}\displaystyle \frac 1{(2j)^2-(2k-1)^2} &\text{ for } j\in J, \\[3mm]
\displaystyle \frac 1{(2k-1)(2j+2k-1)} &\text{ for } j\notin J,\; j\neq n,\\[3mm]
\displaystyle \frac 1{2k-1} &\text{ for } j=n.
 \end{cases}
\end{equation*}
\end{theorem}
\begin{proof}
Using Lemma~\ref{Vorzeichenlemma} we have to deal with the
quantities $\oesgn{M}\oesgn{N}$ in \eqref{real}. Observe that
$\oesgn{N} = \oesgn{\widetilde N}$, since $2n$ is the maximal element.
Thus we may apply Lemma~\ref{Vorzeichenlemma} with $\widetilde N$
instead of $N$ and thus this lemma holds for odd and even
cardinalities of $N$.

In the first and the third sum in the statement of the theorem
expand the determinants simultaneously according to the rows with index in $J.$ The resulting expression is the same as the one we get when we apply Lemma~\ref{Vorzeichenlemma} and Lemma~\ref{Reelllemma} to \eqref{real}. In the second sum of the theorem
we have to expand the determinants simultaneously according to the columns with index in $J$ to finish the proof.
\end{proof}
Setting $s=0$ we get the following expressions for the ``totally
real'' volumes.
\begin{corollary}\label{totalrealvol}
The volumes $ v_{r}^{(0)}$ of the coefficient space of real
polynomials of degree $r$ with all roots real and in the open
unit disk fulfill
\begin{equation*}
\begin{split}
\frac {v_{2n}^{(0)}}{2^{2n}}&=\det \left(\frac
1{(2k-1)(2j+2k-1)}\right)_{1\le j,k\le n}
\qquad\qquad\qquad\text{ if $r=2n$  and}\\[3mm]
\frac {v_{2n-1}^{(0)}}{2^{2n-1}}&=\det \left(\begin{matrix}\frac 1{(2k-1)(2j+2k-1)},& 1\le j \le n-1\\[3mm]
 \frac 1{2k-1},& j=n \end{matrix}\right)_{1\le j,k\le n}
   \qquad\text{ if $r=2n-1.$}
  \end{split}
\end{equation*}
\end{corollary}
Evaluating the determinants by the double alternant
\eqref{doublealternant} or its ``odd'' analogue
\eqref{odddoublealternant}, respectively,  this corollary yields the formul\ae{} for
the totally real volumes known from \cite[Theorem~4.1]{API} (note the close connection to the Selberg integral $S_d(1,1,1/2)$ in this reference).  

\section{The volume ratios in the general case}\label{sec:quo}

Now we are in a position to turn our attention to formul\ae{} for
the quotients $ {v_{d}^{(s)}}/{v_{d}^{(0)}}$ and to confirm Conjecture~\ref{con:integer}. In particular, we
establish the following result involving determinants of partial Hilbert matrices from which we can infer the integrality of $ {v_{d}^{(s)}}/{v_{d}^{(0)}}$ immediately.

\begin{theorem}\label{mixedratio}\label{Integralratios}
The quotients of the volumes \eqref{eq:quotvol} fulfill, for $s\le
\lfloor d/2\rfloor,$
\begin{equation}\label{eq:mixed1}
\begin{split}
\frac {v_{d}^{(s)}}{v_{d}^{(0)}}=\sum_{K\subseteq\{1,\dots,\lfloor d/2\rfloor\}}\binom
{\lfloor d/2\rfloor-|K|}{s-|K|}(-1)^{s+|K|}&\left(\prod_{k\in
K}\frac12\binom{d+2k}{d-2k,\ 2k,\
2k}\right)\prod_{\begin{subarray}{c} j<k\\ j,k\in K\end{subarray}}\frac {(k-j)^2}{(k+j)^2}\\
=\sum_{K\subseteq \{1,\dots,\lfloor d/2\rfloor\}}\binom {\lfloor d/2\rfloor-|K|}{s-|K|}(-1)^{s+|K|}&\left(\prod_{k\in
K}\binom{d+2k}{4k}\binom{4k-1}{2k-1}2k\right)\det \left(\frac
1{j+k}\right)_{j,k\in K}\\
=\sum_{K\subseteq \{1,\dots,\lfloor d/2\rfloor\}}\binom {\lfloor
d/2\rfloor-|K|}{s-|K|}(-1)^{s+|K|}&
\det\Bigg(   
\binom{d+2j}{2j+2k}
\binom{2j+2k-1}{2j-1}
\Bigg)_{j,k\in K}.
\end{split}
\end{equation}
Note that $\binom{n}{k}=0$ for $k<0$ by convention. Moreover, for $K=\emptyset$ the determinant has to be assigned the value $1$. 

In particular, $\frac {v_{d}^{(s)}}{v_{d}^{(0)}}$ is an integer.
\end{theorem}

\begin{proof}
We start with the instance of even degree $d=2n$. Observe that,
setting $X_j=2j$ and $Y_k=2k-1,$ the entries of the matrices
$\big({c}_{j,k}(J)\big)_{1\le j,k\le n}$ from Theorem \ref{Detmix} are of
the form
\begin{equation*}
\begin{split}
&{c}_{j,k}(J) = \begin{cases} \displaystyle\frac 1{X^2_j-Y^2_k} &\text{ for } j\in J, \\[3mm]
 \displaystyle\frac 1{Y_k(X_j+Y_k)} &\text{ for } j\notin J. \end{cases}
\end{split}
\end{equation*}
Using the partial fraction decomposition
\begin{equation}\label{partialfrac}
\begin{split}
\frac 1{X^2_j-Y^2_k} = -\frac{1}{2Y_k}\left(\frac1{X_j+Y_k} +
\frac1{-X_j+Y_k}\right)
\end{split}
\end{equation}
and the multilinearity of the determinant function we
find, since $|J|=s$,
\begin{equation}\label{dK}
\begin{split}
\Delta_J&:= \det \big({c}_{j,k}(J)\big)_{1\le j,k\le n} = \frac
{(-1)^s}{2^s\prod\limits_{1\le k\le n} (2k-1)} \sum_{K\subseteq J}d_K,
\qquad\text{where}\\
d_K&:=\det \big(d_{j,k}(K)\big)_{1\le j,k\le n} \text{ with } d_{j,k}(K) =
\begin{cases} \displaystyle\frac1{-2j+2k-1} &\text{ for } j\in K,\\[3mm]
\displaystyle\frac1{2j+2k-1} &\text{ for } j\notin K. \end{cases}
\end{split}
\end{equation}
Inserting the last result in Theorem~\ref{Detmix} yields
\begin{equation*}
\frac {v_{2n}^{(s)}}{2^{2n+s}}=\sum_{\begin{subarray}{c}J\subseteq \{1,\dots,n\}\\ |J|=s\end{subarray}
}\Delta_J = \frac {(-1)^s}{2^s\prod\limits_{1\le k\le n} (2k-1)}
\sum_{K\subseteq \{1,\dots,n\}}\binom{n-|K|}{s-|K|}d_K.
\end{equation*}
The binomial coefficient occurs due to the following observation. If
$\mathcal{J}_s$ is the collection of all  subsets of cardinality
$s$ of the set $\{1,\ldots, n\}$ then a given set
$K\subseteq \{1,\ldots, n\}$ is a set of exactly
$\binom{n-|K|}{s-|K|}$ elements of the collection $\mathcal{J}_s$.
For the case of $s=0$ nonreal roots this specializes to
\begin{equation*}
\frac {v_{2n}^{(0)}}{2^{2n}}=\frac 1{\prod\limits_{1\le k\le n}
(2k-1)}d_{\emptyset},
\end{equation*}
so that finally
\begin{equation}\label{quotienteven}
\begin{split}
\frac {v_{2n}^{(s)}}{v_{2n}^{(0)}}=(-1)^s\sum_{K\subseteq
\{1,\dots,n\}}\binom{n-|K|}{s-|K|}\frac {d_K}{d_{\emptyset}}
\quad\text{ with $d_K$ from } \eqref{dK} .
\end{split}
\end{equation}
The determinants $d_K$ may be evaluated by the double alternant
formula \eqref{doublealternant} as
\begin{equation*}
d_K=\frac{\prod\limits_{\begin{subarray}{c}j<k \\ j,k\in
K\end{subarray}}2(j-k)\prod\limits_{\begin{subarray}{c}j<k\\
j,k\notin
K\end{subarray}}2(k-j)\prod\limits_{\begin{subarray}{c}j<k\\ j\in K,
k\notin
K\end{subarray}}2(k+j)\prod\limits_{\begin{subarray}{c}j<k\\
j\notin K, k\in
K\end{subarray}}2(-k-j)\prod\limits_{j<k}2(k-j)}{\prod\limits_{ j\in
K, \text{ all }k}(-2j+2k-1)\prod\limits_{ j\notin K, \text{ all
}k}(2j+2k-1)},
\end{equation*}
where all products run over the range $1\le j,k\le n$ with the denoted
restrictions. Using the notation 
$j\equiv k\ (K) \text{ iff } (j,k\in K \text{ or } j,k\notin K),$
this reads
\begin{equation*}
\begin{split}
d_K=&2^{2\binom{n}2}(-1)^{|\{j<k:k\in K, \text{ all
}j\}|}\frac{\prod\limits_{\begin{subarray}{c}j<k\\ j\equiv k (K)\end{subarray}}(k-j){\prod\limits_{\begin{subarray}{c}j<k
\\ j\not\equiv k (K)\end{subarray}}(k+j) }\prod\limits_{j<k}(k-j)}{\prod\limits_{
j\in K, \text{ all }k}(-2j+2k-1)\prod\limits_{ j\notin K, \text{ all
}k}(2j+2k-1)},
\end{split}
\end{equation*}
and, in particular,
\begin{equation*}
\begin{split}
d_{\emptyset}=2^{2\binom{n}2}\frac{\prod\limits_{j<k}(k-j)^2}{\prod\limits_{j,k}(2j+2k-1)},
\end{split}
\end{equation*}
so that the quotients fulfill
\begin{equation}\label{dkd0}
\begin{split}
\frac {d_K}{d_{\emptyset}}= (-1)^{|\{j<k:k\in K, \text{ all }j\}|}
\prod_{\begin{subarray}{c}j<k\\ j\not\equiv k (K)\end{subarray}}\frac{k+j}{k-j}\prod_{ j\in K, \text{
all }k}\frac{2j+2k-1}{-2j+2k-1}.
\end{split}
\end{equation}
For the first product in \eqref{dkd0} we derive
\begin{equation}\label{eq:firstdkd0}
\prod_{\begin{subarray}{c}j<k
\\ j\not\equiv k (K)\end{subarray}}\frac{k+j}{k-j}
=(-1)^{|\{k<j:j\notin K, k\in K\}|}\prod_{k\in
K}\frac{(n+k)!}{k!2k} \frac{(-1)^{n-k}}{(k-1)!(n-k)!}\prod_{\begin{subarray}{c}j\neq k \\ j,k\in
K\end{subarray}}\frac{k-j}{k+j},
\end{equation}
while the second product in \eqref{dkd0} satisfies
\begin{equation}\label{eq:seconddkd0}
\begin{split}
\prod_{ j\in K, \text{ all }k}\frac{2j+2k-1}{-2j+2k-1}=\prod_{ k\in
K, \text{ all }j}\frac{2j+2k-1}{2j-2k-1}=\prod_{k\in
K}\frac{{(2n+2k-1)}!!}{(2k-1)!!^2(2n-2k-1)!!}(-1)^{k}
\end{split}
\end{equation}
with $m!!$ denoting the double factorial.
Therefore, inserting \eqref{eq:firstdkd0} and \eqref{eq:seconddkd0} in \eqref{dkd0} yields
\begin{equation}\label{eq:dkdooo}
\begin{split}
\frac {d_K}{d_{\emptyset}}=&(-1)^{|\{j<k:k\in K, \text{ all
}j\}|+|\{k<j:j\notin K, k\in K\}|}\prod_{k\in K}\frac{(2n+2k)!2^{n-k}2^{2k-1}}{2^{n+k}(2n-2k)!(2k)!^2}(-1)^{n}\prod_{\begin{subarray}{c}j\neq k \\ j,k\in
K\end{subarray}}\frac{k-j}{k+j}\\
=&(-1)^{|\{j<k:k\in K, \text{ all }j\}|+|\{k<j:j\notin K, k\in
K\}|+|\{k<j:j, k\in K\}|}\\ &\cdot\prod_{k\in K}\frac12\binom{2n+2k}{2n-2k,\ 2k,\
2k}(-1)^{n}\prod_{\begin{subarray}{c}j< k \\ j,k\in
K\end{subarray}}\frac{(k-j)^2}{(k+j)^2}.
\end{split}
\end{equation}
Since
\[
\begin{split}
|\{j<k:&k\in K, \text{ all }j\}|+|\{k<j:j\notin K, k\in K\}|+|\{k<j:j, k\in K\}|
\\&
=|\{j\not=k:k\in K, \text{ all }j\}| = |K|(n-1),
\end{split}
\]
equation \eqref{eq:dkdooo} implies that
\begin{equation}\label{eq:enddkdo}
\frac {d_K}{d_{\emptyset}}=(-1)^{|K|}\prod_{k\in
K}\frac12\binom{2n+2k}{2n-2k,\ 2k,\ 2k}\prod_{\begin{subarray}{c}j< k \\ j,k\in
K\end{subarray}}\frac{(k-j)^2}{(k+j)^2}.
\end{equation}
Inserting \eqref{eq:enddkdo} in \eqref{quotienteven} completes the
proof of the first formula for the case of even degree in
Theorem~\ref{mixedratio}. The second formula follows by  writing the
trinomial coefficient as a product of two binomial coefficients
(mind the factor $1/2$ in the product in \eqref{eq:enddkdo}) and
applying the double alternant formula \eqref{doublealternant} to
express the last product by the determinant of the partial Hilbert
matrix from the statement of the theorem.

Let us now turn to the case of odd degree $d=2n-1$. This time we use
the matrices $\big(\tilde c_{j,k}(J)\big)$ from Theorem~\ref{Detmix}. Again
using  the decomposition \eqref{partialfrac} with $X_j=2j$ and
$Y_k=2k-1$
 we get
\begin{equation*}
\begin{split}
\frac {v_{2n-1}^{(s)}}{2^{2n+s-1}}&=\frac
{(-1)^s}{2^s\prod\limits_{1\le k\le n} (2k-1)} \sum_{K\subseteq \{1,\dots,n-1\}}\binom{n-1-|K|}{s-|K|}e_K,\\
\frac {v_{2n-1}^{(0)}}{2^{2n-1}}&=\prod_{1\le k\le n}\frac
1{2k-1}e_{\emptyset},
\end{split}
\end{equation*}
where
\begin{equation}\label{eK}
\begin{split}
e_K= \det(e_{j,k}(K)) \text{ with }
&e_{j,k}(K) = \begin{cases} \displaystyle \frac {1}{-2j+2k-1} &\text{ for } j\in J, \\[3mm]
\displaystyle \frac 1{2j+2k-1} &\text{ for } j\notin J,\ j\neq n,\\[3mm]
\displaystyle 1 &\text{ for } j=n.
  \end{cases}
\end{split}
\end{equation}
Therefore,
\begin{equation}\label{quotientodd}
\begin{split}
\frac {v_{2n-1}^{(s)}}{v_{2n-1}^{(0)}}=(-1)^{s}\sum_{K\subseteq
\{1,\dots,n-1\}}\binom{n-1-|K|}{s-|K|}\frac
{e_K}{e_{\emptyset}} \text{ with $e_K$ from } \eqref{eK} .
\end{split}
\end{equation}
Evaluating the determinants with the variant \eqref{odddoublealternant} of the Cauchy double alternant formula yields
\begin{equation*}
\begin{split}
e_K&=\frac{\prod\limits_{\begin{subarray}{c}j<k \\ j,k\in K
\end{subarray}}2(j-k)\prod\limits_{\begin{subarray}{c}1\le j<k\le n-1 \\
j,k\notin K \end{subarray}}2(k-j)
\prod\limits_{\begin{subarray}{c}1\le j<k \le n-1\\ j\in K, k\notin
K
\end{subarray}}2(k+j)\prod\limits_{\begin{subarray}{c}j<k \\
j\notin K, k\in K \end{subarray}}2(-j-k)\prod\limits_{1\le j<k\le
n}2(k-j)}{\prod\limits_{ j\in K, 1\le k\le
n}(-2j+2k-1)\prod\limits_{\begin{subarray}{c} 1\le j\le n-1, 1\le k\le
n \\ j\notin K \end{subarray}}(2j+2k-1)}\\
&=
2^{\binom{n-1}2+\binom{n}2}(-1)^{|\{j<k: k\in K,\text{ all
}j\}|}\frac{\prod\limits_{\begin{subarray}{c} 1\le j<k\le n-1\\ j\equiv k 
(K)\end{subarray}}(k-j)\prod\limits_{\begin{subarray}{c}1\le j<k\le n-1\\ j\not\equiv k (K)\end{subarray}}(k+j)\prod\limits_{1\le j<k\le n}(k-j)}{\prod\limits_{
j\in K, 1\le k\le n}(-2j+2k-1)\prod\limits_{ j\notin K, 1\le k\le n}(2j+2k-1)},
\end{split}
\end{equation*}
and, in particular,
\begin{equation*}
\begin{split}
e_{\emptyset}=2^{\binom{n-1}2+\binom{n}2}\frac{\prod\limits_{1\le
j<k\le n}(k-j)\prod\limits_{1\le j<k\le
n-1}(k-j)}{\prod\limits_{\begin{subarray}{c}1\le j\le n-1\\ 1\le k\le n \end{subarray}}(2j+2k-1)}.
\end{split}
\end{equation*}
Therefore, the quotients fulfill
\begin{equation}\label{eKe0}
\begin{split}
\frac {e_K}{e_{\emptyset}}= (-1)^{|\{j<k: k\in K,\text{ all }j\}|}
\prod_{\begin{subarray}{c} 1\le j < k\le
n-1 \\ j\not\equiv k (K) \end{subarray}}\frac{k+j}{k-j}\prod_{ j\in K, 1\le k\le
n}\frac{2j+2k-1}{-2j+2k-1}.
\end{split}
\end{equation}
For the the first product in \eqref{eKe0} we get from
\eqref{eq:firstdkd0} by replacing $n$ by $n-1$
\begin{equation*}
\begin{split}
\prod_{\begin{subarray}{c} 1\le j < k\le
n-1 \\ j\not\equiv k (K) \end{subarray}}\frac{k+j}{k-j}=(-1)^{|\{1\le k<j\le n-1:j\notin K, k\in
K\}|}\prod_{k\in K}\frac{(n-1+k)!}{k!(2k)}
\frac{(-1)^{n-1-k}}{(k-1)!(n-1-k)!}\prod_{\begin{subarray}{c}j\neq
k \\ j,k\in K \end{subarray}}\frac{k-j}{k+j},
\end{split}
\end{equation*}
whereas the second product in \eqref{eKe0} equals the result of
\eqref {eq:seconddkd0}.
Therefore, we get
\begin{equation}\label{eq:eeee}
\frac {e_K}{e_{\emptyset}}
=(-1)^{|K|} \prod_{k\in K}\frac12\binom{2n-1+2k}{2n-1-2k,\ 2k,\
2k}\prod_{\begin{subarray}{c} j<k \\ j,k\in K \end{subarray}}\frac{(k-j)^2}{(k+j)^2}.
\end{equation}
Inserting \eqref{eq:eeee} into \eqref{quotientodd} completes the
proof of the first formula for the case $d=2n-1$ in Theorem
\ref{mixedratio}. The second formula follows again by applying the
double alternant formula \eqref{doublealternant} to express the last
product by the determinant of a partial Hilbert matrix. 

The third formula in the theorem follows immediately from 
\begin{align}\label{eq:neueumformung}
\begin{aligned}
\Bigg(\prod_{k\in K}\binom{d+2k}{4k} & \binom{4k-1}{2k-1}2k\Bigg)   \det \left(\frac 1{j+k}\right)_{j,k\in K}
\\
&=\left(\prod_{j\in K}\frac{(d+2j)!}{2(2j-1)!}\prod_{k\in K}\frac{1}{(d-2k)!(2k)!} \right) \det \left(\frac
1{j+k}\right)_{j,k\in K} 
\\
&= \det \left(\frac
{(d+2j)!}{(2j+2k)(d-2k)!(2j-1)!(2k)!}\right)_{j,k\in K} 
\\
&= \det \left( \binom{d+2j}{2j+2k} \binom{2j+2k-1}{2j-1}
 \right)_{j,k\in K}. \qedhere
\end{aligned}
\end{align}
\end{proof}

\begin{remark}\label{eq:mixed2}
Using the second sum for ${v_{2n}^{(s)}}/{2^{2n+s}}$   in Theorem
\ref{Detmix} yields the following alternative formula for the volume
ratios in the case $d=2n$ for all $s\le n.$
\begin{equation*}
\begin{split}
\frac
{v_{2n}^{(s)}}{v_{2n}^{(0)}}&=\sum_{K\subseteq
\{1,\dots,n\}}\binom
{n-|K|}{s-|K|}(-1)^{s-|K|}
\prod_{k\in
K}\frac12\binom{2n+2k-1}{2n-2k+1,\ 2k-1,\
2k-1}\prod_{\begin{subarray}{c} j<k\\ j,k\in K\end{subarray}}\frac {(k-j)^2}{(k+j-1)^2}\\
&\hskip -0.7cm =\sum_{K\subseteq \{1,\dots,n\}}\binom {n-|K|}{s-|K|}(-1)^{s-|K|}
\left(\prod_{k\in
K}\binom{2n+2k-1}{4k-2}\binom{4k-3}{2k-1}(2k-1)\right)\det
\left(\frac 1{j+k-1}\right)_{j,k\in K}
\\
&\hskip -0.7cm =\sum_{K\subseteq \{1,\dots,n\} }\binom
{n-|K|}{s-|K|}(-1)^{s+|K|}\det\Bigg(
\binom{2n+2j-1}{2j+2k-2}\binom{2j+2k-3}{2j-2}
\Bigg)_{j,k\in K}.
\end{split}
\end{equation*}
\end{remark}
We end this section by showing that some results of Kirschenhofer
and Weitzer~\cite{KiWei} as well as Akiyama and Peth\H{o}~\cite{API} are simple consequences of
Theorem \ref{mixedratio}. The first result refers to the instance of
polynomials with exactly one pair of nonreal conjugate zeroes.

\begin{corollary}[{\it cf.}~\cite{KiWei}]
For $d\in\mathbb{N}$ we have
\[
\frac{v_d^{(1)}}{v_d^{(0)}} = \frac{P_d(3)-2d-1}{4},
\]
where $P_d$ denotes the $d$-th Legendre polynomial.
\end{corollary}

\begin{proof}
From \eqref{eq:mixed1} with $s=1$ we have
\begin{equation}\label{even:d1}
\begin{split}
\frac {v_{d}^{(1)}}{v_{d}^{(0)}}=-\lfloor
d/2\rfloor+\sum_{k=1}^{\lfloor
d/2\rfloor}\binom{d+2k}{4k}\binom{4k-1}{2k-1}.
\end{split}
\end{equation}
On the other hand the Legendre polynomials may be defined as
(\cite[p.~66]{Riordan})
\begin{equation}
\begin{split}
P_d(x)=\sum_{j=0}^{d}\binom{d+j}{2j}\binom{2j}{j}\left(\frac{x-1}2\right)^j,
\end{split}
\end{equation}
so that
\begin{equation}\label{P3}
\begin{split}
P_{d}(3)=1+\sum_{j=1}^{d}\binom{d+j}{2j}\binom{2j}{j}.
\end{split}
\end{equation}
Furthermore, since (see {\it e.g.}\ \cite[p.~158]{Rainville})
$P_{d}(-x)=(-1)^d P_d(x),$ we have
\begin{equation*}
\begin{split}
P_{d}(-1)=1+\sum_{j=1}^{d}\binom{d+j}{2j}\binom{2j}{j}(-1)^j=(-1)^d
P_{d}(1)=(-1)^d,
\end{split}
\end{equation*}
yielding
\begin{equation*}
\begin{split}
\sum_{\begin{subarray}{c} j=1 \\[0.5mm] j\text{
odd}\end{subarray}}^d\binom{d+j}{2j}\binom{2j}{j}=1-(-1)^d+\sum_{\begin{subarray}{c}j=1 \\[0.5mm]
j\text{ even} \end{subarray}}^d\binom{d+j}{2j}\binom{2j}{j}.
\end{split}
\end{equation*}
Therefore, setting $j=2k$ for even $j,$ \eqref{P3} can be rewritten as
\begin{equation}
\begin{split}
P_{d}(3)- 2+(-1)^d = 4\sum_{k=1}^{\lfloor d/2\rfloor}\binom{d+2k}{4k}\binom{4k-1}{2k-1},
\end{split}
\end{equation}
so that finally
\begin{equation}
\begin{split}
\frac{P_{d}(3)- 2d -1}4 = -\lfloor d/2\rfloor+\sum_{k=1}^{\lfloor
d/2\rfloor}\binom{d+2k}{4k}\binom{4k-1}{2k-1},
\end{split}
\end{equation}
which coincides with \eqref{even:d1} from above.
\end{proof}

As in \eqref{eq:full} let $v_d$ denote the volume of the coefficient
space $\mathcal{E}_d$ of the polynomials of degree $d$ with all
zeroes in the open unit disk. Summing up over all possible
signatures in Theorem \ref{mixedratio} we get the results on the
quotients $v_d/v^{(0)}_d$ contained in \cite[equation~(9)]{API}.

\begin{corollary}[{{\it cf.}~\cite[equation (9)]{API}}]\label{totalfraceven}
For $d\in\N$ we have
\[
\frac{v_{d}}{v^{(0)}_{d}}=
\begin{cases}
\displaystyle \frac1{2^n}\prod_{j=n+1}^{2n}\binom{2j}j\prod_{j=1}^{n-1}\binom{2j}j^{-1} & \hbox{ for }d=2n,\\[5mm]
\displaystyle\frac
1{2^n}\prod_{j=n}^{2n-1}\binom{2j}j\prod_{j=1}^{n-1}\binom{2j}j^{-1}& \hbox{ for }d=2n-1.
\end{cases}
\]
\end{corollary}

\begin{proof}
Let
\begin{equation}
\label{pnKneu} p_d(K):=\prod_{k\in
K}\binom{d+2k}{4k}\binom{4k-1}{2k-1}\prod_{\begin{subarray}{c}j<k \\ j,k\in K\end{subarray}}\frac
{(k-j)^2}{(k+j)^2}.
\end{equation}
Then by equation \eqref{eq:mixed1} in Theorem~\ref{mixedratio} we have, setting $D=\lfloor d/2 \rfloor$,
\begin{equation*}
\begin{split}
\frac{v_{d}}{v^{(0)}_{d}}&= \sum_{s=0}^{D}
\frac{v^{(s)}_{d}}{v^{(0)}_{d}}= \sum_{s=0}^{D}\sum_{K\subseteq
\{1,\dots,D\}}(-1)^{s-|K|}\binom{D-|K|}{s-|K|}p_d(K)\\
&=\sum_{K\subseteq \{1,\dots,D\}}p_d(K)\sum_{k=0}^{D-|K|}\binom{D-|K|}{k}(-1)^{k}=\sum_{K\subseteq
\{1,\dots,D\}}p_d(K)\delta_{D,|K|}\\
&=p_d(1,\dots,D)=p_d(1,\dots,\lfloor d/2\rfloor).
\end{split}
\end{equation*}
Using \eqref{pnKneu} this yields
\begin{equation*}
\frac{v_{d}}{v^{(0)}_{d}}
 =\frac 1{2^{\lfloor d/2\rfloor}}
\prod_{k=1}^{\lfloor d/2\rfloor}\frac{(d+2k)!(k-1)!^2 k!^2}{(d-2k)!
(2k)!^2(2k-1)!^2}.
\end{equation*}
For  $d=2n$ we get
\begin{equation*}
\frac{v_{2n}}{v^{(0)}_{2n}}
=
\frac1{2^n}\prod_{j=n+1}^{2n}\binom{2j}j\prod_{j=1}^{n-1}\binom{2j}j^{-1},
\end{equation*}
whereas for $d=2n-1$ we find
\begin{equation*}
\frac{v_{2n-1}}{v^{(0)}_{2n-1}} 
=\frac1{2^n}\prod_{j=n}^{2n-1}\binom{2j}j\prod_{j=1}^{n-1}\binom{2j}j^{-1}
\end{equation*}
which concludes the proof.
\end{proof}

The integrality of the ratios ${v_{d}}/{v_{d}^{(0)}}$  is established in \cite{API} by a
careful analysis of the divisors of the numerators and denominators
of the fractions occurring in Corollary~\ref{totalfraceven}. Of
course, the integrality is now immediate from Theorem~\ref{mixedratio} and, furthermore, we gain the following
alternative expressions for the ratios in consideration.

\begin{corollary}
The ratios ${v_{d}}/{v_{d}^{(0)}}$  of the volumes
 from Corollary \ref{totalfraceven} are integers and are given by the alternative
formula
\begin{equation}\label{totalfinalformulaeven}
\frac {v_{d}}{v_{d}^{(0)}}=\det\Bigg(   
\binom{d+2j}{2j+2k}
\binom{2j+2k-1}{2j-1}\Bigg)_{1\le j,k\le
\lfloor d/2\rfloor}.
\end{equation}
\end{corollary}

\begin{proof}
From the proof of Corollary\ref{totalfraceven} we have
\begin{equation*}
\frac{v_{d}}{v^{(0)}_{d}} = p_d(1,2,\dots,\lfloor d/2\rfloor). 
\end{equation*}
But from \eqref{eq:neueumformung} we see that
\[
p_d(1,2,\dots,\lfloor d/2\rfloor)=\det\Bigg(   
\binom{d+2j}{2j+2k}
\binom{2j+2k-1}{2j-1}\Bigg)_{1\le j,k\le
\lfloor d/2\rfloor}. \qedhere
\]
\end{proof}

\section*{Acknowledgements}
We are grateful to the anonymous referees whose comments led to a considerable improvement of 
the exposition of 
this paper. 
In particular, one of the referees suggested to replace our original proof of Corollary~\ref{ConvolutionSdlemma} by introducing $H(X)$ defined in \eqref{eq:newHm} and identity \eqref{eq:nevConv} together with the reference to \eqref{eq:Hclosed} and the idea of using partial fraction techniques.

\bibliographystyle{siam}
\bibliography{ContractPolyR2}

\begin{thebibliography}{10}

\bibitem{Akiyama-Borbeli-Brunotte-Pethoe-Thuswaldner:05}
{\sc S.~Akiyama, T.~Borb{\'e}ly, H.~Brunotte, A.~Peth{\H{o}}, and J.~M.
  Thuswaldner}, {\em Generalized radix representations and dynamical systems.
  {I}}, Acta Math. Hungar., 108 (2005), pp.~207--238.

\bibitem{API}
{\sc S.~Akiyama and A.~Peth{\H{o}}}, {\em On the distribution of polynomials
  with bounded roots, {I}. {P}olynomials with real coefficients}, J. Math. Soc.
  Japan, 66 (2014), pp.~927--949.

\bibitem{APII}
\leavevmode\vrule height 2pt depth -1.6pt width 23pt, {\em On the distribution
  of polynomials with bounded roots {II}. {P}olynomials with integer
  coefficients}, Unif. Distrib. Theory, 9 (2014), pp.~5--19.

\bibitem{AlastueyJankovici:xx}
{\sc A.~Alastuey and B.~Jancovici}, {\em On the classical two-dimensional
  one-component {C}oulomb plasma}, Journal de Physique, 42 (1981), pp.~1--12.

\bibitem{AAR:99}
{\sc G.~E. Andrews, R.~Askey, and R.~Roy}, {\em Special functions}, vol.~71 of
  Encyclopedia of Mathematics and its Applications, Cambridge University Press,
  Cambridge, 1999.

\bibitem{Aomoto:87}
{\sc K.~Aomoto}, {\em Jacobi polynomials associated with {S}elberg integrals},
  SIAM J. Math. Anal., 18 (1987), pp.~545--549.

\bibitem{Cohn:22}
{\sc A.~Cohn}, {\em \"{U}ber die {A}nzahl der {W}urzeln einer algebraischen
  {G}leichung in einem {K}reise}, Math. Z., 14 (1922), pp.~110--148.

\bibitem{Dub:16}
{\sc A.~Dubickas}, {\em Counting integer reducible polynomials with bounded
  measure}, Applicable Analysis and Discrete Mathematics, 10 (2016),
  pp.~308--324.

\bibitem{Fam:89}
{\sc A.~T. Fam}, {\em The volume of the coefficient space stability domain of
  monic polynomials}, Proc. IEEE Int. Symp. Circuits and Systems, 2 (1989),
  pp.~1780--1783.

\bibitem{Fam-Meditch:78}
{\sc A.~T. Fam and J.~S. Meditch}, {\em A canonical parameter space for linear
  systems design}, IEEE Trans. Autom. Control, 23 (1978), pp.~454--458.

\bibitem{FW:08}
{\sc P.~J. Forrester and S.~O. Warnaar}, {\em The importance of the {S}elberg
  integral}, Bull. Amer. Math. Soc. (N.S.), 45 (2008), pp.~489--534.

\bibitem{KLW:15}
{\sc G.~K\'arolyi, A.~Lascoux, and S.~O. Warnaar}, {\em Constant term
  identities and {P}oincar\'e polynomials}, Trans. Amer. Math. Soc., 367
  (2015), pp.~6809--6836.

\bibitem{KT:14}
{\sc P.~Kirschenhofer and J.~M. Thuswaldner}, {\em Shift radix systems---a
  survey}, in Numeration and Substitution 2012, RIMS K\^oky\^uroku Bessatsu,
  B46, Res. Inst. Math. Sci. (RIMS), Kyoto, 2014, pp.~1--59.

\bibitem{KiWei}
{\sc P.~Kirschenhofer and M.~Weitzer}, {\em A number theoretic problem on the
  distribution of polynomials with bounded roots}, Integers, 15, \# A10 (2015).

\bibitem{Krattenthaler2001}
{\sc C.~Krattenthaler}, {\em Advanced determinant calculus}, in The Andrews
  Festschrift: Seventeen Papers on Classical Number Theory and Combinatorics,
  D.~Foata and G.-N. Han, eds., Springer Berlin Heidelberg, Berlin, Heidelberg,
  2001, pp.~349--426.

\bibitem{Macdonald:95}
{\sc I.~G. Macdonald}, {\em Symmetric functions and {H}all polynomials}, Oxford
  Mathematical Monographs, The Clarendon Press, Oxford University Press, New
  York, second~ed., 1995.
\newblock With contributions by A. Zelevinsky, Oxford Science Publications.

\bibitem{RS:02}
{\sc Q.~I. Rahman and G.~Schmeisser}, {\em Analytic theory of polynomials},
  vol.~26 of London Mathematical Society Monographs. New Series, The Clarendon
  Press, Oxford University Press, Oxford, 2002.

\bibitem{Rainville}
{\sc E.~D. Rainville}, {\em Special Functions}, The Macmillan Company, New
  York, 1960.

\bibitem{Riordan}
{\sc J.~Riordan}, {\em Combinatorial identities}, Wiley Series in Probability
  and Mathematical Statistics, John Wiley \& Sons, New York, 1968.

\bibitem{Schur:17}
{\sc J.~Schur}, {\em \"{U}ber {P}otenzreihen, die im {I}nnern des
  {E}inheitskreises beschr\"ankt sind. {I}}, J. Reine Angew. Math., 147 (1917),
  pp.~205--232.

\bibitem{Schur:18}
\leavevmode\vrule height 2pt depth -1.6pt width 23pt, {\em {\"U}ber
  {P}otenzreihen, die im {I}nnern des {E}inheitskreises beschr\"ankt sind.
  {II}}, J. Reine Angew. Math., 148 (1918), pp.~122--145.

\bibitem{Selberg:44}
{\sc A.~Selberg}, {\em Remarks on a multiple integral}, Norsk Mat. Tidsskr., 26
  (1944), pp.~71--78.

\bibitem{Todd:60}
{\sc J.~Todd}, {\em Computational problems concerning the {H}ilbert matrix}, J.
  Res. Nat. Bur. Standards Sect. B, 65B (1961), pp.~19--22.

\end{thebibliography}
\end{document}